\documentclass[a4paper,dvipdfm,11pt]{amsart}
\usepackage{amsmath}
\usepackage{amssymb}
\usepackage{enumerate}
\usepackage{amscd}
\usepackage{graphics}
\usepackage{latexsym}
\usepackage{verbatim} 
\usepackage{url}
\usepackage{bbm}

\theoremstyle{plain}
\newtheorem{thm}{{\bf Theorem}}[section]

\newtheorem{prop}[thm]{{\bf Proposition}}
\newtheorem{lemma}[thm]{{\bf Lemma}}
\newtheorem{fact}[thm]{{\bf Fact}}

\theoremstyle{definition}
\newtheorem{define}[thm]{{\bf Definition}}

\newtheorem{question}[thm]{{\bf Question}}

\newcommand{\cf}{\mathord{\mathrm{cf}}}

\newcommand{\dom}{\mathord{\mathrm{dom}}}
\newcommand{\size}[1]{\left\vert {#1} \right\vert}
\newcommand{\p}{\mathcal{P}}

\newcommand{\col}{\mathord{\mathrm{Col}}}

\newcommand{\seq}[1]{\langle {#1} \rangle}

\newcommand{\ka}{\kappa}
\newcommand{\la}{\lambda}
\newcommand{\om}{\omega}

\newcommand{\bbP}{\mathbb{P}}
\newcommand{\bbQ}{\mathbb{Q}}

\newcommand{\calF}{\mathcal{F}}
\newcommand{\calG}{\mathcal{G}}

\newcommand{\calU}{\mathcal{U}}

\newcommand{\HOD}{\mathrm{HOD}}

\newcommand{\ZF}{\mathsf{ZF}}
\newcommand{\ZFC}{\mathsf{ZFC}}

\newcommand{\AC}{\mathsf{AC}}
\newcommand{\fix}{\mathrm{fix}}
\newcommand{\sym}{\mathrm{sym}}
\newcommand{\HS}{\mathrm{HS}}
\newcommand{\Fn}{\mathrm{Fn}}

\title[A note on uniform ultrafilters in a choiceless context]{
A note on uniform ultrafilters in a choiceless context}
\author[T. Usuba]{Toshimichi Usuba}
\address[T. Usuba]
{Faculty of Science and Engineering,
Waseda University, 
Okubo 3-4-1, Shinjyuku, Tokyo, 169-8555 Japan}
\email{usuba@waseda.jp}
\begin{document}
\maketitle
\begin{abstract}
In \cite{HK}, Hayut and Karagila asked some questions
about uniform ultrafilters in a choiceless context.
We provide several answers to their questions.
\end{abstract}
\section{Introduction}
A proper ultrafilter $U$ over an infinite cardinal $\ka$ is said to be \emph{uniform}
if every element of $U$ has cardinality $\ka$.
Let $\calU$ be the class of all infinite cardinals
$\ka$ which carries a uniform ultrafilter.
In $\ZFC$, $\calU$ has a trivial structure: It is just the class of all cardinals.
However it is not the case if the Axiom of Choice fails.
\begin{thm}[Hayut-Kagagila \cite{HK}]
Relative to a certain large cardinal assumption,
it is consistent that $\ZF+$``\,$\aleph_0, \aleph_\om \in \calU$ but $\aleph_{n+1} \notin \calU$ for every $n<\om$''.
Oppositely, it is also consistent that
$\ZF+$``\,$\aleph_\om \notin \calU$ but $\aleph_{n+1} \in \calU$ for every $n<\om$''.
\end{thm}
Furthermore, Hayut and Karagila demonstrated 
that the behavior of $\calU$ at the successors of regular cardinals 
can be manipulated as you like.
With these results, they asked the following questions for singular cardinals and its successors.

\begin{question}[\cite{HK}]
\begin{enumerate}
\item Is it consistent for $\aleph_{\om+1}$ to be the least element of $\calU$?
More generally, what behavior is consistent at successors of singular cardinals?
\item Is it consistent for a singular cardinal, and specifically $\aleph_\om$, to
be the least cardinal not in $\calU$?
\item Assume there is a uniform ultrafilter on $\aleph_{\om_\om}$, does that imply
there is a uniform ultrafilter on $\aleph_{\om}$? 
Or more generally, if $\la>\cf(\la)$ carries
a uniform ultrafilter, does that imply that any other singular cardinal with the
same cofinality carries a uniform ultrafilter?
\end{enumerate}
\end{question}

In this paper, we provide several answers to these questions
by proving  the following theorems:
\begin{thm}\label{1.3}
Relative to a certain large cardinal assumption,
it is consistent that
$\ZF+$``\,$\aleph_{\om+1}$ is the
least element of $\calU$''.
\end{thm}

\begin{thm}\label{1.4}
It is consistent that
$\ZF+$``\,$\aleph_{\om}$ is the
least cardinal not in $\calU$''.
\end{thm}

\begin{thm}\label{1.5}
Relative to a certain large cardinal assumption,
it is consistent that
$\ZF+$``\,$\aleph_{\om}$ is the
least cardinal not in $\calU$''$+$``every singular cardinal $>\aleph_\om$ with countable cofinality 
is in $\calU$''.
\end{thm}

In \cite{HK}, they also asked the following:
\begin{question}[\cite{HK}]
Is it consistent that $\ka$ does not carry a uniform ultrafilter, $\ka^+$ does,
but $\ka^+$ is not measurable, and is this possible without using large cardinals? In
particular, is it consistent that $\aleph_0$ is the only measurable cardinal, while $\aleph_1\notin \calU$
and $\aleph_2 \in \calU$?
\end{question}
While we do not have a full answer to this question,
we prove that such a situation has a large cardinal strength.
\begin{thm}[In $\ZF$]\label{1.6}
If there are cardinals $\ka<\la$ with $\ka \notin \calU$ but $\la \in \calU$,
then  there is an inner model of a measurable cardinal.
\end{thm}
This theorem also shows that
large cardinal assumptions in Theorems \ref{1.3} and \ref{1.5}
cannot be eliminated.

\section{Preliminaries}
Throughout this paper, we always suppose that every successor cardinal is regular.
First we prove basic lemmas which will be used later.
The following lemmas follow from the standard arguments,
but here we present choiceless proofs for the completeness.
Let $\bbP$ be a poset with maximum element $\mathbbm 1$.
For a set $x$, let $\check x$ be a canonical name for $x$,
namely, $\check x=\{\seq{\check y, \mathbbm 1} \mid y \in x\}$. 

\begin{lemma}[In $\ZF$]
Let $\bbP$ be a countable poset.
Then $\bbP$ preserves all cofinalities and cardinals.
\end{lemma}
\begin{proof}
For cofinality, it is enough to show that if $\ka$ is regular uncountable,
then $\Vdash$``$\check \ka$  is regular''.
To verify this,
take $p \in \bbP$, $\alpha<\ka$, and
a name $\dot f$ such that $p \Vdash \dot f:\check \alpha \to \check \ka$.
For each $\beta<\alpha$,
let $A_\beta=\{\xi \mid \exists q \le p (q \Vdash \dot f(\check \beta)=\check \xi)\}$.
We know $p \Vdash \dot f (\check \beta) \in \check A_\beta$,
and since $\bbP$ is countable,
we have that $A_\beta$ is countable.
Let $\gamma=\sup\{\sup A_\beta \mid \beta<\alpha\}$.
Since $\ka$ is regular uncountable,
we have $\gamma<\ka$ and  $p \Vdash \dot f``\alpha \subseteq \gamma$.
Hence $\dot f$ is not forced to be a cofinal map.

For preserving cardinals,
take cardinals $\ka<\la$, $p \in \bbP$, and a name $\dot f$ for a function from $\ka$ to $\la$.
Define $F:\bbP \times \ka \to \la$ as follows:
If $q \Vdash \dot f(\check \alpha)=\check \beta$ for some $\beta<\la$,
set $F(q, \alpha)=\beta$, here note that such $\beta$ is unique for $q$ and $\alpha$.
Otherwise, let $F(q,\alpha)=0$.
Since $\bbP$ is countable, we have $\size{\bbP \times \ka}=\ka$.
Hence $F$ cannot be a sujection and we can take $\gamma<\la$ with
$\gamma \notin \mathrm{range}(F)$.
Then we have $p \Vdash \check \gamma \notin \mathrm{range}(\dot f)$, so $\dot f$ is forced to be non-surjective.
\end{proof}

For a cardinal $\ka$, an ultrafilter $U$ is \emph{$\ka$-complete}
if for every $\alpha<\ka$ and $f:\alpha \to U$,
we have $\bigcap f``\alpha \in U$.
$U$ is \emph{$\sigma$-complete} if it is $\om_1$-complete.

\begin{lemma}[In $\ZF$]\label{1.9}
If there is a non-principal non-$\sigma$-complete ultrafilter $U$ over a set $S$,
then $\om$ carries a non-principal ultrafilter.
\end{lemma}
\begin{proof}
Since $U$ is not $\sigma$-complete,
we can find a function $f:\om \to U$ such that
$\bigcap f``n \in U$ for every $n<\om$ but
$\bigcap f``\om =\emptyset$.
Define $g:S \to \om$ as follows:
For $s \in S$, let $g(s)$ be the least $n<\om$ with $s \notin \bigcap f``n$.
Then the family $g_*(U)=\{X \subseteq \om \mid g^{-1}(X) \in U\}$ is 
a non-principal ultrafilter over $\om$. 
\end{proof}

\begin{lemma}[In $\ZF$]\label{1.10}
Suppose the Countable Choice holds.
Let $\bbP$ be a countable poset, $\ka$ a cardinal,
and $\dot U$  a name such that
$\Vdash$``\,$\dot U$ is a $\sigma$-complete uniform ultrafilter over $\check \ka$''.
Then there is $p \in \bbP$ such that the following hold:
\begin{enumerate}
\item For every $X \subseteq \ka$, either
$p \Vdash \check X \in \dot U$ or $p \Vdash \check \ka \setminus \check X \in \dot U$.
\item The set 
$\{X \subseteq \ka \mid p \Vdash \check X \in \dot U\}$ is a $\sigma$-complete uniform ultrafilter over $\ka$.
\end{enumerate}
\end{lemma}
\begin{proof}
(2) is immediate from (1).
For (1), suppose not.
By the Countable Choice,
we can find $\{X_p \mid p \in \bbP\}$ such that $X_p \subseteq \ka$,
$p \not \Vdash \check X_p \in \dot U$, and
$p \not \Vdash \check \ka \setminus \check X_p \in \dot U$.
Take a name $ \dot Y$ such that
\[
\Vdash \dot Y=\bigcap \{ \check X_p \mid p \in \bbP, \check X_p \in \dot U\}
\cap 
\bigcap \{ \check \ka \setminus \check X_p \mid p \in \bbP, \check \ka \setminus \check X_p \in \dot U\}.
\]

Since $\bbP$ is countable, we have $\Vdash \dot Y \in \dot U$.
We also know $\Vdash$``$\dot Y \subseteq \check X_p$ or $\dot Y \cap \check X_p=\emptyset$''
for every $p \in \bbP$.

For $p \in \bbP$, let $Y_p=\{\alpha <\ka \mid p \Vdash \check \alpha \in \dot Y\}$.
We have $\Vdash \dot Y=\bigcup\{\check Y_p \mid p \in \dot G\}$ where $\dot G$ is a canonical name
for a generic filter.
Again, since $\bbP$ is countable and $\dot U$ is forced to be $\sigma$-complete,
we can find $p, q \in \bbP$ such that
$q \Vdash \check p \in \dot G \land \check Y_p \in \dot U$.
By extending $q$, we may assume $q \le p$.
Then $Y_q \supseteq Y_p$, so $q \Vdash \check Y_q \in \dot U$.
Because $q \Vdash \check Y_q \subseteq \dot Y$,
if $Y_q \cap X_q \neq \emptyset$ then $q \Vdash \check X_q \in \dot U$,
this contradicts the choice of $X_q$.
Hence $Y_q \cap X_q=\emptyset$,
but then $q \Vdash \check \ka \setminus \check X_q \in \dot U$,
this is also a contradiction.
\end{proof}

We will use \emph{symmetric extensions}.
Here we review it, and see Jech \cite{J} for details.
We emphasize that we do not need $\AC$ for taking symmetric extensions and
establishing basic results about it.

Every automorphism $\pi$ on $\bbP$
induces the isomorphism $\pi$ on the $\bbP$-names,
namely, $\pi(\dot x)=\{\seq{\pi(\dot y), \pi(p)} \mid \seq{\dot y, p } \in \dot x\}$
for a $\bbP$-name $\dot x$.
\begin{fact}
Let $p \in \bbP$, 
$\varphi$ be a formula of set theory,
and $\dot x_0,\dotsc, \dot x_n$ $\bbP$-names.
Let $\pi$ be an automorphism on $\bbP$.
Then $p \Vdash \varphi(\dot x_0,\dotsc, \dot x_n)$ if and only if
$\pi(p) \Vdash \varphi(\pi(\dot x_0),\dotsc, \pi(\dot x_n))$.
\end{fact}

Let $\mathcal{G}$ be a subgroup of the automorphism group on $\bbP$.
A non-empty family $\calF$ of subgroups of $\calG$ is 
a \emph{normal filter on $\calG$} if the following hold:
\begin{enumerate}
\item If $H \in \calF$ and $H'$ is a subgroup of $\calG$ with
$H \subseteq H'$, then $H' \in \calF$.
\item For $H, H' \in \calF$, we have $H \cap H' \in \calF$.
\item For every $H \in \calF$ and $\pi \in \calG$,
the set $\pi^{-1} H\pi=\{\pi^{-1} \circ \sigma \circ \pi \mid \sigma \in H\}$ is in $\calF$.
\end{enumerate}
A triple $\seq{\bbP, \calG, \calF}$ is called a \emph{symmetric system}.

For a $\bbP$-name $\dot x$,
let $\mathrm{sym}(\dot x)=\{\pi \in \calG \mid \pi(\dot x)=\dot x\}$,
which is a subgroup of $\calG$.
A name $\dot x$ is \emph{symmetric}
if $\mathrm{sym}(\dot x) \in \calF$,
and \emph{hereditarily symmetric}
if $\dot x$ is symmetric and
for every $\seq{\dot y, p} \in \dot x$,
$\dot y$ is hereditarily symmetric.
\begin{fact}
If $\dot x$ is a hereditarily symmetric name and $\pi \in \calG$,
then $\pi(\dot x)$ is also hereditarily symmetric.
\end{fact}


Let $\mathrm{HS}$ be the class of all hereditarily symmetric names.
For a $(V, \bbP)$-generic $G$, let $\mathrm{HS}^G$ be the class  of all interpretations of
hereditarily symmetric names by $G$.
$\mathrm{HS}^G$ is a transitive model of $\ZF$ with
$V \subseteq \mathrm{HS}^G \subseteq V[G]$.
$\mathrm{HS}^G$ is called a \emph{symmetric extension of $V$}.

\section{$\aleph_{\om+1}$ can be the least elemet of $\calU$}
We give a proof of Theorem \ref{1.3}.
For this sake, we use the following Apter and Madigor's theorem.
Recall that an uncountable cardinal $\ka$ is \emph{measurable}
if $\ka$ carries a $\ka$-complete non-principal ultrafilter.
In $\ZF$, every measurable cardinal is regular.

\begin{thm}[Apter \cite{A}, Apter-Magidor \cite{AM}]
Suppose $V$ satisfies $\mathsf{GCH}+$``\,$\ka<\la$ are such that
$\ka$ is supercompact and $\la$ is the least measurable cardinal above $\ka$''.
Then there is a symmetric extension $N$ 
such that the following hold in $N$:
\begin{enumerate}
\item $\mathsf{DC}_{\aleph_\om}$ holds in $N$.
\item $\la=\aleph_{\om+1}$ is measurable in $N$.
\item The cardinal and cofinality structure  $\ge \la$ is the same as in $V$.
\end{enumerate}
\end{thm}

We start from this Apter and Magidor's model,
that is, we work in a model $V$ of $\ZF+\mathsf{DC}_{\aleph_\om}+$
``\,$\aleph_{\om+1}$ is measurable''$+$``every successor cardinal is regular''.

\begin{lemma}\label{3.2}
For $n<\om$, there is no $\sigma$-complete uniform ultrafilter over $\aleph_n$.
\end{lemma}
\begin{proof}
Suppose not, and take the least $n<\om$
such that $\aleph_n$ carries a $\sigma$-complete uniform ultrafilter $U$.
Clearly $n>0$. First we show that
$U$ is not $\aleph_{n}$-complete.
By $\mathsf{DC}_{\aleph_\om}$,
we can take a 1-1 sequence $\seq{r_\alpha \mid \alpha<\aleph_n}$
of subsets of $\aleph_{n-1}$.
For $\beta<\aleph_{n-1}$,
define $A_\beta \in U$ as follows:
If $\{\alpha<\aleph_n \mid \beta \in r_\alpha\} \in U$,
then put $A_\beta=\{\alpha<\aleph_n \mid \beta \in r_\alpha\}$.
Otherwise, that is, 
if $\{\alpha<\aleph_n \mid \beta \notin r_\alpha\} \in U$,
then $A_\beta=\{\alpha<\aleph_n \mid \beta \notin r_\alpha\}$.
If $U$ is $\aleph_n$-complete,
we have $\bigcap_{\beta<\aleph_{n-1}} A_\beta \in U$.
Pick $\alpha, \alpha' \in \bigcap_{\beta<\aleph_{n-1}} A_\beta$ with
$\alpha<\alpha'$.
By the choice of the $A_\beta$'s, we have $r_\alpha=r_{\alpha'}$.
This is a contradiction.

Now we know $U$ is not $\aleph_{n}$-complete.
Take the largest $m<n$ such that
$U$ is $\aleph_m$-complete.
Then we can find $f:\aleph_m \to U$ such that
$\bigcap f``\alpha \in U$ for every $\alpha<\aleph_m$
but $\bigcap f``\aleph_m=\emptyset$.
Define $g:\aleph_n \to \aleph_m$ as that
$g(\beta)$ is the least $\alpha<\aleph_m$ with
$\beta \notin \bigcap f``\alpha$.
Consider the ultrafilter $g_*(U)=\{X \subseteq \aleph_m \mid g^{-1}(X) \in U\}$.
By the choice of $g$, one can check that $g_*(U)$ is a uniform ultrafilter over $\aleph_m$,
moreover it is $\sigma$-complete.
This contradicts to the minimality of $n$.
\end{proof}

For a set $X$, let $\Fn(X, 2)$ be the poset of all finite partial functions from $X$ to $2$ with
the reverse inclusion.
We define a symmetric system $\seq{\bbP, \calG, \calF}$, which is Feferman's one.

Let $\bbP$ be the poset $\Fn(\om \times \om, 2)$.
For a set $A \subseteq \om \times \om$,
let $\pi_A$ be the automorphism $\pi_A$ on $\bbP$ defined as follows:
For $p \in \bbP$, $\dom(\pi_A(p))=\dom(p)$,
and 
\[
\pi_A(p)(m,n)=\begin{cases} 1-p(m,n) & \text{ if $\seq{m,n} \in A$},\\
p(m,n) & \text{ if $\seq{m,n} \notin A$.}
\end{cases}
\]
Let $\calG$ be the set $\{\pi_A \mid A \subseteq \om \times \om\}$.
$\calG$ is a subgroup of the automorphism group of $\bbP$.
For $m<\om$,
let $\fix (m)=\{\pi_A \in \calG \mid A \cap (m \times \om)=\emptyset\}$,
$\fix(m)$ is a subgroup of $\calG$.
Let $\calF=\{H \subseteq \calG \mid H$ is a subgroup of $\calG,
\fix(m) \subseteq H$ for some $m<\om\}$.
It is routine to check that $\calF$ is a normal filter on $\calG$.

Take a $(V, \bbP)$-generic $G$.
Notice that $\bbP$ is countable,
hence $\bbP$ preserves all cofinalities and cardinals.
In particular every successor cardinal is regular in $V[G]$.
Consider a symmetric extension $\HS^G$.

The following Lemmas \ref{3.3}--\ref{3.5} are known (see \cite{HK}),
but we present proofs for the completeness. Again,
we do not need $\AC$ for proving these lemmas.

For $p \in \bbP$ and $m<\om$,
let $p \restriction m$ be the condition $p \restriction (m \times \om)$, it is in $\Fn(m \times \om,2)$.
\begin{lemma}\label{3.3}
Let $\varphi(v_0,\dotsc, v_n)$ be a formula of set theory
and $\dot x_0,\dotsc \dot x_n$ be $\bbP$-names.
Let $m<\om$, and suppose $\fix(m) \subseteq \sym(\dot x_i)$ for every $i \le n$.
If $p \Vdash \varphi(\dot x_0,\dotsc, \dot x_n)$,
then $p \restriction m \Vdash \varphi(\dot x_0,\dotsc, \dot x_n)$.
\end{lemma}
\begin{proof}
Suppose not, and take $q \le p \restriction m$ such that
$q\Vdash \neg \varphi(\dot x_0,\dotsc, \dot x_n)$.
Let $A=\{\seq{n, i} \in \dom(q) \cap \dom(p) \mid q(n,i) \neq p(n,i)\}$.
We know $A \cap (m \times \om)=\emptyset$,
and so $\pi_A \in \fix(m)$ and $\pi_A(\dot x_k)=\dot x_k$ for $k \le n$.
Moreover $\pi_A(q)$ is compatible with $p$, but
$\pi_A(q) \Vdash \neg \varphi(\dot x_1,\dotsc, \dot x_n)$.
This is a contradiction.
\end{proof}

For $m<\om$, let $x_m=\{n <\om \mid \exists p \in G(p(m,n)=1)\}$,
and $\dot x_m$ be the name $\{\seq{\check m, p} \mid p(m,n)=1\}$.
$\dot x_m$ is a canonical hereditarily symmetric name for $x_m$,
so we have $x_m \in \HS^G$.

\begin{lemma}\label{3.4}
In $\HS^G$,
there is no non-principal ultrafilter over $\om$.
\end{lemma}
\begin{proof}
Suppose there is a non-principal ultrafilter $U$ over $\om$.
Fix a hereditarily symmetric name $\dot U$ for $U$.
Fix $m<\om$ with $\fix(m) \subseteq \sym(\dot U)$.
We prove that both $x_m$ and $\om \setminus x_m$ are not in $U$,
this is a contradiction.

Suppose $x_m \in U$.
Take $p \in G$ such that $p \Vdash \dot x_m \in \dot U$.
Fix a large $n_0<\om$ such that
$\dom(p) \cap (\{m\} \times \om )\subseteq \{m\} \times n_0$,
and let $A=\{m\} \times [n_0, \om)$.
One can check that $\pi_A(p)=p$, and 
 $\Vdash \pi_A(\dot x_m) \cap \dot x_m \subseteq \check n_0$.
Since $p \Vdash \dot x_m \in \dot U$,
we have $p \Vdash \pi_A(\dot x_m) \in \pi_A(\dot U)=\dot U$,
hence $p \Vdash \dot x_m \cap \pi_A(\dot x_m) \subseteq \check n_0 \in \dot U$.
This is a contradiction.
The case $\om \setminus x_m \in U$ follows from a similar argument.
\end{proof}

Since every uniform ultrafilter over $\aleph_\om$ is not $\sigma$-complete,
we also have:
\begin{lemma}
In $\HS^G$,
there is no uniform ultrafilter over $\aleph_\om$.
\end{lemma}
\begin{proof}
Otherwise, we can take a non-principal ultrafilter over $\om$ by Lemma \ref{1.9},
this contradicts to Lemma \ref{3.4}.
\end{proof}

Since $\aleph_{\om+1}$ is measurable in $V$,
we can fix an $\aleph_{\om+1}$-complete non-principal ultarfilter $U \in V$ over $\aleph_{\om+1}$,
which is a uniform ultrafilter.
Since $\bbP$ is countable and $U$ is $\sigma$-complete in $V$, one can check that:
\begin{lemma}\label{3.5}
In $V[G]$, the set $\{X \subseteq \aleph_{\om+1} \mid \exists Y \in U (Y \subseteq X)\}$
is a uniform ultrafiler over $\aleph_{\om+1}$.
In particular,
in $\mathrm{HS}^G$, $U$ generates a uniform ultrafilter over $\aleph_{\om+1}$.
\end{lemma}
\begin{proof}
It is enough to check that for every $X \subseteq \aleph_{\om+1}$,
there is $Y \in U$ such that $Y \subseteq X$ or $X \cap Y=\emptyset$.
Take $p \in \bbP$ and a name $\dot X$ for a subset of $\aleph_{\om+1}$.
For $q \le p$, let $Y_q=\{\alpha<\aleph_{\om+1} \mid q \Vdash \check \alpha \in \dot X\}$.
If $Y_q \in U$ for some $q \le p$,
then $q \Vdash \check Y_q \subseteq \dot X$.
If $Y_q \notin U$ for every $q \le p$,
let $Y=\bigcap_{q \le p}( \aleph_{\om+1} \setminus Y_q)$.
We have $Y \in U$ since $U$ is 
$\sigma$-complete and $\bbP$ is countable.
In addition we have $p \Vdash \check Y \cap \dot X=\emptyset$.
\end{proof}

For $m<\om$, let $G_m =G \cap \Fn(m \times \om, 2)$.
$G_m$ is $(V, \Fn(m \times \om, 2))$-generic.
The name $\{\seq{(p \restriction m\check ), p} \mid  p\in \bbP\}$ is a canonical hereditarily 
symmetric name for $G_m$.
Hence we have:
\begin{lemma}
$V[G_m] \subseteq \HS^G$ for every $m<\om$.
\end{lemma}

Finally we prove that $\aleph_n$ does not carry a uniform ultrafilter over $\aleph_n$ for every $n<\om$ in
$\HS^G$.

\begin{lemma}
In $\mathrm{HS}^G$, for every $n<\om$,
there is no uniform ultrafilter over $\aleph_n$.
\end{lemma}
\begin{proof}
Suppose not, and let $n<\om$ and $W$ a uniform ultrafilter over $\aleph_n$.
We know $n>0$ by Lemma \ref{3.4}.
If $W$ is not $\sigma$-complete in $\mathrm{HS}^G$,
then we can derive a non-principal ultrafilter over $\om$ by Lemma \ref{1.9},
this is a contradiction.
Hence $W$ is $\sigma$-complete.

Take a hereditarily symmetric name $\dot W$ for $W$ and $m<\om$
such that $\fix(m) \subseteq \sym(\dot W)$.
For each $X \in \p(\aleph_n) \cap V[G_m]$,
there is a hereditarily symmetric name $\dot X$ for $X$ with
$\fix(m) \subseteq \sym(\dot X)$.
Hence  by Lemma \ref{3.3}, we have:
\[
X \in W  \iff p \Vdash \dot X \in \dot W \text{ for some $p \in G_m$.}
\]
Thus $W'=W \cap V[G_m] \in V[G_m]$,
which is a $\sigma$-complete ultrafilter over $\aleph_n$ in $V[G_m]$.
Then by Lemma \ref{1.10}, we can find a $\sigma$-complete uniform ultrafilter over $\aleph_n$ in $V$,
this contradicts to Lemma \ref{3.2}.
\end{proof}

\section{$\aleph_\om$ can be the least cardinal not in $\calU$}
We start the proof of Theorem \ref{1.4}.
For a set $X$ and a cardinal $\ka$,
let $\Fn(X,2, \mathop{<}\ka)$ be the poset of all
partial functions $p:X \to 2$ with size $<\ka$.
The following lemma is well-known:
\begin{lemma}
Let $\ka$ be a regular uncountable cardinal and
$\bbP, \bbQ$ posets.
If $\bbP$ satisfies the $\ka$-c.c. and $\bbQ$ is $\ka$-closed,
then $\Vdash_\bbP$``\,$\check \bbQ$ is $\ka$-Baire''.
\end{lemma}

Suppose GCH. 
For $n<\om$, let $\bbQ_n=\mathrm{Fn}(\aleph_\om, 2, \mathop{<}\aleph_n)$.
Let $\bbP$ be the full support product of the $\bbQ_n$'s, that is,
$\bbP=\prod_{n<\om} \bbQ_n$, and $p \le q \iff p(n) \le q(n)$ in $\bbQ_n$ for every $n<\om$.
For simplicity, we denote $p(n)(\alpha)$ as $p(n,\alpha)$.

Let $n<\om$ and $\ka=\aleph_{n+1}$.
The poset $\bbP$ can be factored as
the product $(\prod_{m \le n} \bbQ_n) \times (\prod_{n<m<\om} \bbQ_m)$.
The poest $\prod_{n<m<\om} \bbQ_m$ is $\ka$-closed,
and, by GCH, $\prod_{m \le n} \bbQ_n$ satisfies the $\ka$-c.c.
Thus $\prod_{m\le n} \bbQ_n$ forces that $\prod_{n<m<\om} \bbQ_m$ is $\ka$-Baire.
By using this observation,
one can check that 
$\bbP$ preserves all cofinalities and cardinals.

For $A \subseteq \om \times \aleph_\om$,
define the automorphism $\pi_A$ on $\bbP$ as follows:
$\dom(\pi_A(p) (n))=\dom(p(n))$ for every $n<\om$, and
\[
\pi_A(p)(n,\alpha)=\begin{cases}
1-p(n,\alpha) & \text{ if $\seq{n,\alpha} \in A$,}\\
p(n,\alpha) & \text{ if $\seq{n, \alpha} \notin A$}.
\end{cases}
\]
Let $\calG=\{\pi_A \mid A \subseteq \om \times \aleph_\om\}$,
this is a subgroup of the automorphism group of $\bbP$.
For $n<\om$,
let $\fix(n)=\{\pi_A \mid (n \times \aleph_\om) \cap A=\emptyset\}$.
$\fix(n)$ is a subgroup of $\calG$.
Let $\calF=\{H \subseteq \calG \mid H$ is a subgroup with $\fix(n) \subseteq H$ for some $n<\om\}$.
One can check that $\calF$ is a normal filter on $\calG$.



Take a $(V, \bbP)$-generic $G$.
For $n<\om$, let $G_n=\{p \restriction n \mid p \in G\}$
which is $(V, \prod_{m<n} \bbQ_m)$-generic.
Let $\dot G_n=\{ \seq{ (p\restriction n\check), p} \mid p \in \bbP\}$.
$\dot G_n$ is a hereditarily symmetric name for $G_n$ with
$\fix(n) \subseteq \sym(\dot G_n)$.
Hence we have:
\begin{lemma}
$G_n \in \HS^G$,
in particular $V[G_n] \subseteq \HS^G$
for every $n<\om$.
\end{lemma}

\begin{lemma}
For every $n<\om$, $\aleph_n$ carries a uniform ultrafilter over $\aleph_n$ in $\HS^G$.
\end{lemma}
\begin{proof}
Fix $n<\om$.
$\bbP$ can be factored as
the product $(\prod_{m \le n} \bbQ_n) \times (\prod_{n<m<\om} \bbQ_m)$,
and, in $V[G_{n+1}]$, the poset
$(\prod_{n<m<\om} \bbQ_m)$ is $\aleph_{n+1}$-Baire.
Hence $\p(\aleph_n)^{V[G]}=\p(\aleph_n)^{V[G_{n+1}]}$.
Because $V[G_{n+1}]$ satisfies $\AC$,
we can find a uniform ultrafilter over $U$ in $V[G_{n+1}]$,
and $U$ remains an ultrafilter in $V[G]$.
Since $V[G_{n+1}] \subseteq \HS^G \subseteq V[G]$,
we have $U \in \HS^G$ is a uniform ultrafilter in $\HS^G$. 
\end{proof}

For each $n<\om$,
let $X_n=\{\alpha<\aleph_\om \mid p(n,\alpha)=1$ for some $p \in G\}$,
and $\dot X_n$ be the $\bbP$-name $\{\seq{\check \alpha, p} \mid p \in \bbP, p(n,\alpha)=1\}$.
$\dot X_n$
is a canonical hereditarily symmetric name for $X_n$ with 
$\fix(n+1) \subseteq \sym(\dot X_n)$.
Hence $X_n \in \HS^G$.
It is routine to check the following:
\begin{lemma}
Let $n<\om$ and $A \subseteq \om \times \aleph_\om$.
Let $x=\{\alpha<\aleph_\om \mid \seq{n, \alpha} \notin A\}$.
Then 
$\Vdash \dot X_n \cap \pi_A(\dot X_n) \subseteq \check x$.
\end{lemma}

\begin{lemma}
There is no uniform ultrafilter over $\aleph_\om$ in $\HS^G$.
\end{lemma}
\begin{proof}
Suppose not.
Let $U \in \HS^G$ be a uniform ultrafilter over $\aleph_\om$,
and $\dot U$ be a hereditarily symmetric name for $U$.
Take $n<\om$ with $\fix(n) \subseteq \sym(\dot U)$.
We show that both $X_n$ and  $\aleph_\om \setminus X_n$ are not in $U$.

First suppose to the contrary that $X_n \in U$.
Take $p \in G$ such that $p \Vdash \dot X_n \in \dot U$.
Let $d=\dom(p(n))$, and $A = \{\seq{n,\alpha} \in \om \times \aleph_\om \mid 
\alpha \notin d\}$.
We have $\pi_A \in \fix(n)$ and $\pi_A(p)=p$.
Moreover $\Vdash \dot X_n \cap \pi_A(\dot X_n) \subseteq \check d$.
Since $\pi_A \in \fix(n) \subseteq \sym(\dot U)$,
we have $p \Vdash \pi_A(\dot X_n) \in \pi_A(\dot U)=\dot U$,
so we have $p \Vdash \dot X_n \cap \pi_A(\dot X_n) \subseteq \check d \in \dot U$.
However this is impossible since $\size{d}<\aleph_n$.
The case that $\aleph_\om \setminus X_n \in U$ follows from a similar argument.
\end{proof}

This completes the proof of Theorem \ref{1.4}.
Our proof is flexible;
We can prove the following by a similar argument.
The proof is left to the reader.
\begin{thm}
Suppose GCH. Let $\alpha$ be a limit ordinal.
Then there is a  cardinal preserving symmetric extension
in which $\aleph_\alpha$ is the least cardinal not in $\calU$.
\end{thm}

Let us say that a cardinal $\ka$ is \emph{strong limit}
if for every $\alpha<\ka$, there is no surjection from $\p(\alpha)$ onto $\ka$.
In the resulting model of Theorem \ref{1.4},
$\aleph_\om$ is not strong limit.
\begin{question}
Is it consistent that
$\aleph_\om$ is strong limit and  the least cardinal not in $\calU$?
\end{question}

\section{$\aleph_\om \notin \calU$ but $\la \in \calU$ for every $\la>\aleph_\om$ with countable cofinality}
In this section we give a proof of Theorem \ref{1.5}.

Suppose GCH, and there is a strongly compact cardinal $\ka$.
Every regular cardinal $\ge \ka$ carries a $\ka$-complete uniform ultrafilter.

Let $\bbP$ be the poset from the previous section,
and let $\col(\aleph_{\om+1}, \mathop{<}\ka)$ be the  standard $\aleph_{\om+1}$-closed Levy collapsing poset
which force $\ka$ to be $\aleph_{\om+2}$.
Take a $(V, \col(\aleph_{\om+1}, \mathop{<}\ka))$-generic $H$.
In $V[H]$, by Hayut-Karagila's symmetric collapse argument (\cite{HK}),
we can find a symmetric extension $M$ of $V$ in which the following hold:
\begin{enumerate}
\item $\ka=\aleph_{\om+2}$.
\item For every regular cardinal $\la \ge \ka$,
every $\ka$-complete uniform ultrafilter $U$ over $\la$ in $V$ generates
a $\ka$-complete uniform ultrafilter.
\item $\p(\aleph_\om)^M=\p(\aleph_\om)^V$.
\end{enumerate}
Next take a $(V[H], \bbP)$-generic $G$.
Since $\col(\aleph_{\om+1}, \mathop{<}\ka)$ is $\aleph_{\om+1}$-closed,
$\bbP$ preserves all cardinals between $V[H]$ and $V[H][G]$,
in particular $\ka=\aleph_{\om+2}$ in $V[G][H]$.
Then, take a symmetric extension $N$ of $M[G]$ via $\bbP$ as in the previous section.
We show that $N$ is a required model.

By the argument before, we can show that
in $N$, $\aleph_\om$ is the least cardinal not in $\calU$.
In addition, since $\bbP$ has cardinality $\aleph_{\om+1}$ in $V$ (and so in $M$),
for every regular cardinal $\la \ge \ka$, 
a $\ka$-complete uniform ultrafilter over $\la$ in $M$ generates 
an ultrafilter in $N$.

\begin{lemma}
In $N$, every singular cardinal $\la>\aleph_\om$ with countable cofinality carries a uniform ultrafilter.
\end{lemma}
\begin{proof}
Fix a singular cardinal $\la>\aleph_\om$ with countable cofinality.
Note that $\cf(\la)=\om$ in $V$.
In $V$, fix an increasing sequence $\seq{\la_n \mid n<\om}$ of
regular cardinals with limit $\la$,
and fix also $\seq{U_n \mid n<\om}$ such that
each $U_n$ is a $\ka$-complete uniform ultrafilter over $\la_n$.

In $N$, for each $n<\om$,
$U_n$ generates an ultrafilter.
Fix a non-principal ultrafilter $U$ over $\om$,
and define a filter $W$ over $\la$ as:
\[
X \in W \iff \{n<\om \mid X \cap \la_n \in U_n\} \in U.
\]
It is easy to check that $W$ is a uniform ultrafilter over $\la$.
\end{proof}
By a similar argument, one can prove that every cardinal $>\aleph_{\om+1}$ with cofinality not equal to $\aleph_{\om+1}$ carries a uniform ultrafilter in $N$, however $\aleph_{\om+1}$ would not.
\begin{question}
Is it consistent that $\aleph_\om$ is the unique cardinal not in $\calU$?
\end{question}

\section{Consistency strength about $\calU$}
To prove Theorem \ref{1.6}, we need the notions of indecomposable ultrafilter and
regular ultrafilter.
\begin{define}
Let $U$ be an ultrafilter over a set $S$.
\begin{enumerate}
\item Let $\ka$ be a cardinal.
$U$ is said to be \emph{$\ka$-indecomposable} if
for every $f:S \to \ka$,
there is $X \in [\ka]^{<\ka}$ such that $f^{-1}(X) \in U$.
\item For cardinals $\ka \le \la$,
$U$ is said to be \emph{$(\ka,\la)$-regular}
if there is a family $\{A_\alpha \mid \alpha<\la\} \subseteq U$
such that $\bigcap_{\alpha \in x}A_\alpha=\emptyset$ for every $x \in [\la]^{\ka}$.
\end{enumerate}
\end{define}

The existence of  non-regular ultrafilters is a large cardinal property.
\begin{thm}[Donder \cite{D}]\label{5.1}
If there are cardinals $\ka<\la$ such that
$\la$ carries a uniform ultrafilter which is not $(\om, \ka)$-regular,
then there is an inner model of a measurable cardinal.
\end{thm}

The following lemmas are  kind of folklore.
\begin{lemma}
Let $U$ be an ultrafilter over a set $S$
and $\ka$  a cardinal.
Then the following are equivalent:
\begin{enumerate}
\item $U$ is $(\om, \ka)$-regular.
\item There is a family $\{B_s \mid s \in S\}$ such that
$B_s \in [\ka]^{<\om}$ and $\{s \in S \mid \alpha \in B_s\} \in U$ for every
$\alpha<\ka$.
\end{enumerate}
\end{lemma}
\begin{proof}
(1) $\Rightarrow$ (2).
Fix a family $\{A_\alpha \mid \alpha<\ka\} \subseteq U$ witnessing that $U$ is $(\om ,\ka)$-regular.
For each $s \in S$, let $B_s=\{\alpha<\ka \mid s \in A_\alpha\}$.
By the choice of the $A_\alpha$'s, we have that $B_s$ is finite.
Moreover, for $\alpha<\ka$, we have
$\{s \in S \mid \alpha \in B_s\}=A_\alpha \in U$.

(2) $\Rightarrow$ (1). 
Let $A_\alpha=\{s \in S \mid \alpha \in B_s\} \in U$ for $\alpha<\ka$.
For $x \in [\ka]^\om$, if there is $s \in \bigcap_{\alpha \in x} A_\alpha$,
then $B_s$ is infinite, this is a contradiction.
\end{proof}

\begin{lemma}
Let $U$ be an ultrafilter over a set $S$
and $\ka$  a cardinal.
If $U$ is $\ka$-indecomposable,
then $U$ is not $(\om, \ka)$-regular.
\end{lemma}
\begin{proof}
Suppose not.
By the previous lemma,
we can find a family 
$\{B_s \mid s \in S\}$ such that
$B_s \in [\ka]^{<\om}$ and $\{s \in S \mid \alpha \in B_s\} \in U$ for every
$\alpha<\ka$.
Since $U$ is $\ka$-indecomposable and $\ka^{<\om}=\ka$,
there is $X \subseteq  [\ka]^{<\om}$ such that
$\size{X}<\ka$ and 
$\{s \in S \mid B_s \in X\} \in U$.
Because $\size{\bigcup X}<\ka$,
we can pick $\alpha \in \ka \setminus \bigcup X$.
Then there must be
$s \in S$ such that $B_s \in X$ but 
$\alpha \in B_s$, this is a contradiction.
\end{proof}

By this lemma and Donder's theorem,
we have:
\begin{prop}\label{5.1}
If there are cardinals $\ka<\la$ such that $\la$ carries a
$\ka$-indecomposable uniform ultrafilter,
then there is an inner model of a measurable cardinal.
\end{prop}

We start the proof of Theorem \ref{1.6}.

\begin{thm}[In $\ZF$]
If there are cardinals $\ka<\la$ with $\ka \notin \calU$ but $\la \in \calU$,
then there is an inner model of a measurable cardinal.
\end{thm}
\begin{proof}
Fix a uniform ultrafilter $U$ over $\la$.
Then $U$ must be $\ka$-indecomposable;
If not, there is $f:\la \to \ka$ such that
$f^{-1}(X) \notin U$ for every $X \in [\ka]^{<\ka}$.
Then the ultrafilter $f_*(U)=\{X \subseteq \ka \mid f^{-1}(X) \in U\}$
forms a uniform ultrafilter over $\ka$, so $\ka \in \calU$.
This is a contradiction.

We note that, we cannot apply Donder's theorem  at the moment,
since $V$ may not satisfy the Axiom of Choice.
To settle this matter, we have to take an inner model of $\ZFC$ and work in it.

Here we recall some definition.
Let $\mathrm{OD}[U]$ be the class of all sets which are definable with parameters $U$ and ordinals,
and $\HOD[U]$ be of all $x$ with $\mathrm{trcl}(\{x\}) \subseteq \mathrm{OD}[U]$.
$\HOD[U]$ is a transitive model of $\ZFC$ with
$U'=U \cap \HOD[U] \in \HOD[U]$.
$U'$ is a uniform ultrafilter over $\la$ in $\HOD[U]$.
We use the following Vop\v enka's theorem:
For every set $X$ of ordinals,
there is a poset $\bbP \in \HOD[U]$ and
a $(\HOD[U], \bbP)$-generic $G \in V$ with
$X \in \HOD[U][G]$.
See, e.g., Woodin-Davis-Rodriguez \cite{WDR} for the proof.
Note that, for proving 
 Vop\v enka's theorem,
$V$ does not need to satisfy $\AC$.

Case 1: $\ka$ is regular.
We see that $U'$ is $\ka$-indecomposable in $\HOD[U]$.
To do this, take $f:\la \to \ka$ with $f \in \HOD[U]$.
By the assumption and the regularity of $\ka$,
there is $\alpha<\ka$ such that
$f^{-1}(\alpha) \in U$, hence $f^{-1}(\alpha) \in U'$.

We have that $U'$ is $\ka$-indecomposable in $\HOD[U]$.
Because $\HOD[U]$ is a model of $\ZFC$,  $\HOD[U]$ has an inner model of a measurable cardinal
by Proposition \ref{5.1}.

Case 2: $\ka$ is singular.
Now suppose to the contrary that
there is no inner model of a measurable cardinal.
In this case, we use Dodd-Jensen core model $K^{\HOD[U]}$ of $\HOD[U]$,
which is a forcing invariant definable transitive model of $\ZFC$.
By the assumption that no inner model of a measurable cardinal,
$K^{\HOD[U]}$ satisfies the covering theorem for $\HOD[U]$,
that is, for every set $x \in \HOD[U]$ of ordinals,
there is $y \in K^{\HOD[U]}$ such that
$x \subseteq y$ and $\size{y}^{\HOD[U]} \le \max(\aleph_1^{\HOD[U]}, \size{x}^{\HOD[U]})$.
For details of Dodd-Jensen core model, see Jech \cite{J1}.

Again, we see that $U'$ is $\ka$-indecomposable in $\HOD[U]$.
Take a function $f \in \HOD[U]$ from $\la$ to $\ka$.
By the assumption,
there is $X \in [\ka]^{<\ka}$ such that
$f^{-1}(X) \in U$. 
By Vop\v enka's theorem,
we can find a poset $\bbP \in \HOD[U]$ and
a $(\HOD[U], \bbP)$-generic $G \in V$ with
$X \in \HOD[U][G]$.
Since $\HOD[U][G]$ does not have an inner model of a measurable cardinal,
$K^{\HOD[U][G]}$ satisfies the covering theorem for $\HOD[U][G]$.
So there is $Y \in K^{\HOD[U][G]}$
such that $X \subseteq Y \subseteq \ka$ and
$\size{Y}^{\HOD[U][G]} \le \max(\aleph_1^{\HOD[U][G]}, \size{X}^{\HOD[U][G]})$.
We know $K^{\HOD[U]}=K^{\HOD[U][G]}$, so $Y \in K^{\HOD[U]} \subseteq \HOD[U]$.
Clearly $f^{-1}(Y) \in U$, and we have to check $\size{Y}^{\HOD[U]}<\ka$.
Since $\HOD[U][G]$ satisfies $\AC$, we have that
$\ka>\aleph_1^{\HOD[U][G]}$.
Thus we have $\size{Y}<\ka$ in $\HOD[U]$.

Now we know that $U'$ is a $\ka$-indecomposable uniform ultrafilter over $\la$ in $\HOD[U]$.
Since $\HOD[U]$ is a model of $\ZFC$, we have that 
there is an inner model of a measurable cardinal 
by Proposition \ref{5.1}. This is a contradiction.
\end{proof}

\end{document}